\documentclass{amsart}

\usepackage{amssymb}
\usepackage{tikz}
\usepackage{tikz-cd}
\usepackage{graphicx}

\usepackage{color}

\newtheorem{thm}{Theorem}[section]
\newtheorem{cor}[thm]{Corollary}
\newtheorem{lem}[thm]{Lemma}
\newtheorem{prop}[thm]{Proposition}

\theoremstyle{definition}
\newtheorem{defn}[thm]{Definition}
\newtheorem{exm}[thm]{Example}
\newtheorem{remark}[thm]{Remark}

\newcommand{\m}{\mathbf{m}}
\newcommand{\kk}{\mathbb{K}}
\newcommand{\cS}{\mathcal{S}}
\newcommand{\J}{\mathcal{J}}
\newcommand{\frakm}{\mathfrak{m}}
\DeclareMathOperator{\syz}{syz}

\def\ZZ{\mathbb{Z}}

\def\dd{\partial}

\def\cal{\mathcal}  % ams compatibility
\newcommand{\A}{{\cal A}}

\numberwithin{equation}{section}

\newcommand{\lt}{L_2^{\text{\tiny{trip}}}}

\begin{document}

\title[]{Free Multiplicities on the moduli of $X_3$}


%    author one information

\author{Michael DiPasquale}
\address{Department of Mathematics, Oklahoma State University, Stillwater, OK, 74078 USA}
\curraddr{}
\email{mdipasq@okstate.edu}

\author{Max Wakefield}
\address{Department of Mathematics, United States Naval Academy, Annapolis, MD, 21402 USA}
\curraddr{}
\email{wakefiel@usna.edu}

%    author two information
%\author{}
%\address{}
%\curraddr{}
%\email{}
%\thanks{The author has been supported by the Office of Naval Research.}

%    \subjclass is required.
%\subjclass[]{52C35 (55R80)}

%\subjclass{}
%\subjclassyear{2010}

\date{}

\dedicatory{}

%    Abstract is required.
\begin{abstract} In this note we study the freeness of the module of derivations on all moduli of the $X_3$ arrangement with multiplicities. We use homological techniques stemming from work of Yuzvinsky, Brandt, and Terao which have recently been developed for multi-arrangements by the first author, Francisco, Mermin, and Schweig.

\end{abstract}

\maketitle


%--------------------------------- Introduction -------------------------------------%

\section{Introduction}
%\bigskip

Fix a field $\kk$ of characteristic zero and let $S=\kk [x,y,z]$ be the polynomial ring on three variables which we associate to the symmetric algebra of a three dimensional vector space $V$. The $X_3$ arrangement is the collection of the hyperplanes $H_1=\{ x=0\}$, $H_2=\{ y=0\}$, $H_3=\{z=0\}$, $H_4=\{x+y=0\}$, $H_5=\{x+z=0\}$, and $H_6=\{y+z=0\}$ in $V$. The associated matroid is called the rank 3 whirl (see \cite{Ox} Appendix and pictured in Figure \ref{X3pic}) and is the unique relaxation of the braid matroid whose free multiplicities were studied in \cite{DFMS16}. A multiplicity on $X_3$ is a function $\m :X_3 \to \ZZ_{>0}$ and let $\m=[m_1,\ldots ,m_6]$ denote the multiplicity vector on $\A$ where $\m(H_i)=m_i$ for all $1\leq i\leq 6$ respectively. In Figure \ref{X3pic} we show a real projective picture of $X_3$ with the corresponding multiplicities.

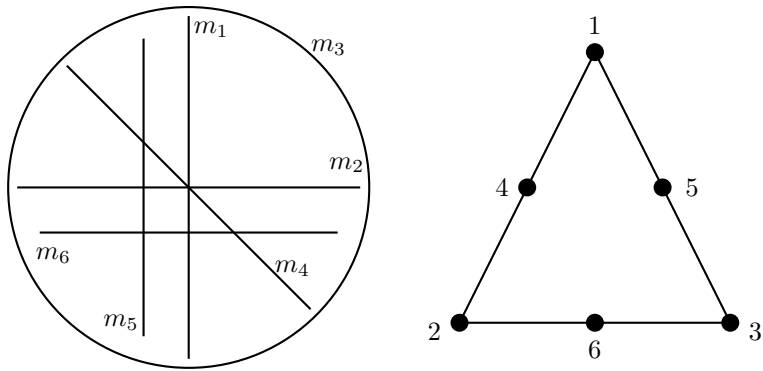
\begin{figure}

\begin{tikzpicture}[scale=.6]

\draw[thick] (0,0) circle (4cm);
\draw[thick] (-3.8,0)--(3.8,0);
\draw[thick] (0,-3.8)--(0,3.8);
\draw[thick] (-2.7,2.7)--(2.7,-2.7);
\draw[thick] (-1,3.3)--(-1,-3.3);
\draw[thick] (3.3,-1)--(-3.3,-1);
\node at (.5,3.5) {$m_1$};
\node at (3.5,.5) {$m_2$};
\node at (3.1,3.1) {$m_3$};
\node at (2.3,-1.8) {$m_4$};
\node at (-1.5,-3) {$m_5$};
\node at (.-3,-1.5) {$m_6$};

\draw[thick,fill] (6,-3) circle (5pt);
\draw[thick,fill] (9,-3) circle (5pt);
\draw[thick,fill] (12,-3) circle (5pt);
\draw[thick,fill] (9,3) circle (5pt);
\draw[thick,fill] (7.5,0) circle (5pt);
\draw[thick,fill] (10.5,0) circle (5pt);
\draw[thick] (6,-3)--(12,-3)--(9,3)--(6,-3);

\node[below, left] at (5.8,-3.2) {2};
\node[below] at (9,-3.2) {6};
\node[below,right] at (12.2,-3.2) {3};
\node[above] at (9,3.2) {1};
\node[left] at (7.3,0) {4};
\node[right] at (10.8,0) {5};

\end{tikzpicture}
\caption{A projective picture of the $X_3$ arrangement labeled with multiplicities and its associated matroid diagram}\label{X3pic}
\end{figure}

The problem of classifying free multiplicities on a free arrangement is usually very difficult (in Section \ref{defs} we cover the background material for freeness). At the moment there are only four main results where arrangements (or classes of arrangements) free multiplicities are classified. 

\begin{enumerate}

\item In \cite{ATY09} Abe, Yoshinaga, and Terao  show that the only arrangements with the property that all multiplicities are free are products of one or two dimensional arrangements.

\item In \cite{Y10} Yoshinaga shows that generic arrangements have no free multiplicities.

\item In \cite{A07} Abe classifies the free multiplicities on the deleted $A_3$ braid arrangement.

\item In \cite{DFMS16} the first author, Francisco, Mermin, and Schweig finish the classification of free multiplicities on the full $A_3$ braid arrangement started in \cite{ATW08} by Abe, Terao, and the second author and progressed in \cite{ANN09} by Abe, Nuida, and Numata.

\end{enumerate}

The idea of this note is to be a positive addition to this list.  It turns out that the $X_3$ arrangement has a one dimensional moduli space (the realization space of the rank 3 whirl matroid, see Section \ref{defs} for details). The main result, Theorem~\ref{thm:X3Classification}, is not only a classification of the free multiplicities on $X_3$ but also a classification of the free multiplicities on all of its moduli.  Since the characteristic polynomial of $X_3$ does not factor, no arrangement in the moduli of $X_3$ is free (by Terao's factorization theorem~\cite{Te81}).  This is the first non-generic arrangement to have it's free multiplicities classified where (1) it's not free and (2) it has non-trivial moduli.  One consequence of this classification is many multiplicities where freeness of the associated multi-arrangement is not combinatorially determined, contrary to Terao's conjecture for simple arrangements (see Example~\ref{ex:FreeNotComb}).

%NExt paragraph added by mike 06/15
Our proof of the classification follows along the lines of two recent papers~\cite{D17,DFMS16} and also draws on work of Brandt and Terao~\cite{BT94} and Schenck and Stillman~\cite{SS96}.  We prove that freeness of $D(X_3,\m)$ is characterized by the vanishing of a certain homology module of a chain complex.  The relevant chain complex is a modification (to the setting of multi-arrangements) of a chain complex appearing in work of Brandt and Terao~\cite{BT94}, where it is used to study formality.  On the other hand, our analysis of this chain complex closely mirrors techniques in multivariate spline theory~\cite{SS96}.  The techniques are not isolated to the $X_3$ arrangement and can be generalized to arbitrary multi-arrangements; this is the intended subject of a forthcoming paper.

Our paper is arranged as follows.  In Section \ref{defs} we review basic definitions and enumerate the moduli of $X_3$. In section \ref{proofs} we introduce the homological machinery necessary for our classification and in section~\ref{sec:freemult} we classify all the free multiplicities on $X_3$. In section \ref{exams} we discuss free extensions of $X_3$; in particular we show that extensions of $X_3$ and its moduli satisfy Terao's conjecture.

{\bf Acknowledgements:} The authors want to thank Graham Denham for pointing out some free multiplicities on $X_3$ which was the main motivation for this project. We also want to thank Takuro Abe and Masahiko Yoshinaga for pointing out the use of the multiarrangement addition-deletion theorem with the deleted $A_3$ arrangement. The second author is partially supported by the Simons Foundation and the Office of Naval Research. 

\section{Set up}\label{defs}

In this section we first review the basic definitions of free multiplicities. Then we discuss the moduli on $X_3$.

\subsection{Derivations on multiarrangements}

The module of derivations on $S$ is $\mathrm{Der}(S)=\{ \theta \in \mathrm{Hom}(S,S)\ |\ \forall f,g\in S,\ \theta (fg)=\theta(f)g+f\theta (g)\}$ and the module of derivations on the multiarrangement $(\A,\m)$ where $\A=\{ H_i\}$, with defining linear forms $\alpha_i$, is $$D(\A,\m)=\{\theta\in \mathrm{Der}(S)\ |\ \theta (\alpha_i)\in \alpha_i^{\m(H_i)} S\}. $$
 
\begin{defn} A multiplicity $\m$ on an arrangement $\A$ is \emph{free} for $\A$ if $D(\A,\m)$ is a free module over the polynomial ring $S$.  If $\m$ is a free multiplicity we say $(\A,\m)$ is free.
\end{defn}

Write $Q(\A,\m)=\prod_{H\in\A} \alpha_H^{\m(H)}$.  We may refer to a multi-arrangement by its polynomial $Q(\A,\m)$.  The following is Saito's critierion for multi-arrangements.

\begin{prop}\cite{Z89}
Suppose $(\A,\m)$ is a multi-arrangement and $\theta_1,\ldots,\theta_\ell\in D(\A,\m)$.  Write $\theta_i=\sum_{j} \theta_{ij}\dd_j$ in the basis elements $\dd_j=\partial/\partial x_j$ and let $M=M(\theta_1,\ldots,\theta_\ell)$ be the $\ell\times\ell$ matrix with entries $\theta_{ij}$.  The following are equivalent.
\begin{enumerate}
\item $D(\A,\m)$ is a free $S$-module with basis $\theta_1,\ldots,\theta_\ell$.
\item $\det(M)=kQ(\A,\m)$ for some $k\neq 0\in\kk$.
\end{enumerate}
\end{prop}

If $\A\subset V\cong\kk^\ell$ is an arrangement and $H\in\A$ a hyperplane, the restriction $\A^H$ is the hyperplane arrangement in $H\cong \kk^{\ell-1}$ with hyperplanes $\{H\cap H': H'\in\A\}$.  The \textit{Zeigler multiplicity} on $\A^H$ is $\m(H')=\#\{H''\in\A: H''\cap H=H'\cap H\}$ and the Ziegler multi-restriction of $\A$ to $H\in\A$ is the multi-arrangement $(\A^H,\m)$ where $\m$ is the Ziegler multiplicity.  If $\A$ is a central arrangement, denote by $L$ its lattice of flats; this is the poset of all intersections of hyperplanes of $\A$ ordered with respect to reverse inclusion.  If $X\in L$ then $\A_X$ is the arrangement consisting of all hyperplanes containing $X$.  The following criterion is due to Yoshinaga~\cite{Y04}.

\begin{thm}\label{thm:Yosh}
Suppose $\A$ is a central arrangement with maximal flat $c$ and $H\in\A$.  Then $\A$ is free if and only if
\begin{enumerate}
\item The Ziegler multi-restriction $(\A^H,\m)$ is free and
\item $\A_X$ is free for every $X\neq c\in L$ with $X>H$.
\end{enumerate}
\end{thm}

\subsection{Moduli of $X_3$}

Consider the coefficient matrix $M=(a_{ij})$ where $H_i=\{a_{i1}x+a_{i2}y+a_{i3}z=0\}$ for $1\leq i\leq 6$. Since hyperplanes $H_1$, $H_2$, $H_5$ and $H_6$ form a generic subarrangement of $X_3$ we can fix coordinates so that the coefficient matrix is of the form $$M=\left[ \begin{array}{ccc}
1&0 &0\\
0&1&0\\
a_{31}&a_{32}&a_{33}\\
a_{41}&a_{42}&a_{43}\\
0&0&1\\
1&1&1\\
\end{array}\right]$$ Then we get the following consequences of the three triple points:

\begin{enumerate}

\item The triple point $\{H_1,H_3,H_5\}$ gives that $a_{32}=0$. 

\item The triple point $\{H_2,H_3,H_6\}$ gives that $a_{31}=a_{33}$.

\item The triple point $\{H_1,H_2,H_4\}$ gives that $a_{43}=0$. 
\end{enumerate} Since $a_{31}\neq 0$ and $a_{41}\neq 0$ (otherwise we would have more dependences than are combinatorially allowed for $X_3$) we can scale the rows so that the coefficient matrix is $$M=\left[ \begin{array}{ccc}
1&0 &0\\
0&1&0\\
1&0&1\\
1&a_{42}&0\\
0&0&1\\
1&1&1\\
\end{array}\right] .$$ Since we do not have any more multiple intersection points this is as far as we can reduce $M$. Hence the dimension of the moduli space of $X_3$ is 1. Moreover, by examining all other possible determinants of $M$ we see that as long as $a_{42}\neq 0, 1$ the corresponding matroid is isomorphic to that of $X_3$. If $a_{42}=0$ then $H_1=H_4$. If $a_{42}=1$ then we get a new dependency which makes the arrangement lattice equivalent to the braid arrangement. 

Again changing coordinates and using $\alpha\neq 0,1$ for the one dimensional moduli we will use the following coefficient matrix for the remainder of the paper $$M=\left[ \begin{array}{ccc}
1&0 &0\\
0&1&0\\
0&0&1\\
1&-\alpha&0\\
1&0&1\\
0&1&1\\
\end{array}\right] $$ which can be viewed as tilting the ``diagonal line'' in Figure \ref{pic2}.

\begin{figure}

\begin{tikzpicture}[scale=.6]

\draw[thick] (0,0) circle (4cm);
\draw[thick] (-3.8,0)--(3.8,0);
\draw[thick] (0,-3.8)--(0,3.8);
\draw[thick, red] (-3.5,1.5)--(3.5,-1.5);
\draw[thick] (-1,3.3)--(-1,-3.3);
\draw[thick] (3.3,-1)--(-3.3,-1);
\node at (.5,3.5) {$m_1$};
\node at (3.5,.5) {$m_2$};
\node at (3.1,3.1) {$m_3$};
\node at (2.3,-1.8) {$m_4$};
\node at (-1.5,-3) {$m_5$};
\node at (.-3,-1.5) {$m_6$};

\draw[->,red,thick] (-1.2,2.4) to [out=180, in=70] (-2.7,1.3);
\node[red] at (-2,1.8) {$\alpha$};

\end{tikzpicture}
\caption{The $X_3$ arrangement with the moduli emphasized}\label{pic2}
\end{figure}
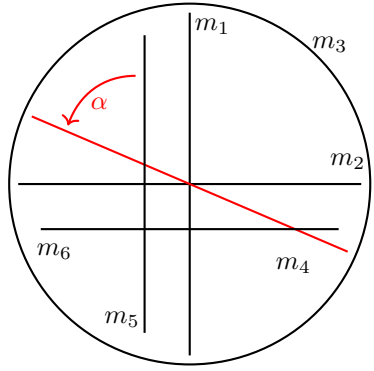

%EXample added by mike 06/13

\begin{exm}[Freeness of multi-arrangements is not combinatorial]\label{ex:FreeNotComb}
Consider the multiplicity $\m=[3,3,3,1,1,1]$ on the moduli of the $X_3$ arrangement over the complex numbers, with corresponding multi-arrangement defined by $x^3y^3z^3(x-\alpha y)(x+z)(y+z)$.  The multi-arrangement defined thus is free if and only if $\alpha\neq -1$, which shows that freeness of multi-arrangements is not a combinatorial property.

We can prove this claim using addition-deletion techniques for multi-arrangements from~\cite{ATW08} (if the reader is not familiar with these techniques, feel free to skip the following argument).  First, consider the deletion $(\A,\m')$ with defining polynomial $x^3y^3z^3(x+z)(y+z)$.  Then $\A$ is supersolvable with filtration $\A=\A_3\supset\A_2\supset\A_1$, where $\A_1,\A_2,\A_3$ have defining polynomials $x$,$xz(x+z)$, and $xyz(x+z)(y+z)$, respectively.  We can check that $(\A,\m')$ is free with exponents $(3,4,4)$ by~\cite[Theorem~5.10]{ATW08}.  The Euler restriction (see~\cite{ATW08}) of $(X_3,\m)$ to $V(x)=H_1$ is the multi-arrangement $x^3z^3(x+z)(x+\alpha z)$ (some care must be taken to see that $3$ is the proper exponent on $x$, but this follows from~\cite[Proposition~4.1]{ATW08}).  Ziegler considered the multi-arrangement $x^3z^3(x+z)(x+\alpha z)$ to show that exponents of multi-arrangements are not combinatorial~\cite[Proposition~10]{Z89}; he showed that the exponents are $(3,5)$ if $\alpha=-1$ and $(4,4)$ if $\alpha\neq -1$.  It then follows from the addition theorem~\cite[Theorem~0.7]{ATW08} that the multi-arrangement $x^3y^3z^3(x-\alpha y)(x+z)(y+z)$ is free if $\alpha\neq -1$ and from the restriction theorem~\cite[Theorem~0.6]{ATW08} that it is not free when $\alpha=-1$.
\end{exm}

%REmark added by mike 06/13

\begin{remark}\label{rem:P1}
A significant obstruction to using addition-deletion techniques from~\cite{ATW08} to classify all free multiplicities on the moduli of $X_3$ is that exponents for multi-arrangements of four points in $\mathbb{P}^1$ are largely unknown.  If the multiplicity vector $\m$ of an arrangement $\A$ of points in $\mathbb{P}^1$ satisfies certain inequalities and the points in $\A$ are `generic' then the second author and Yuzvinsky show that the exponents of $(\A,\m)$ are $(\lfloor |\m|/2\rfloor,\lceil |\m|/2 \rceil)$~\cite{WY07}.  However it is not easy to determine what makes the arrangement of points in $\mathbb{P}^1$ `generic'; it depends in particular on the multiplicity vector $\m$. 

%We will discuss another restriction to using the multiarrangement addition theorem from \cite{ATW08} in Remark \ref{}.
\end{remark}

\section{Homological characterization of freeness}\label{proofs}

In this section we prove that freeness of $(D(X_3),\m)$ can be characterized by the vanishing of a certain homology of a chain complex.  We devote this section to a careful description of the chain complex and a proof that the vanishing of this particular homology is both necessary and sufficient for freeness of $D(X_3,\m)$.  The techniques are close in spirit to~\cite{SS96}, where it is shown that vanishing homologies of a chain complex determine freeness of the module of splines.

We will denote by $L=L_{X_3}$ the intersection lattice of $X_3$, and we denote by $L_c$ intersections of codimension (or rank) $c$.  The $X_3$ arrangement has three triple points $Y_1,Y_2,$ and $Y_3$ given as intersections of $\{H_1,H_2,H_4\},\{H_1,H_3,H_5\},$ and $\{H_2,H_3,H_6\}$, respectively.  Let $\lt=\{Y_1,Y_2,Y_3\}$ be the set of triple points.  For every hyperplane $H_i$ in $X_3$ set $J(H_i)=\alpha_i^{\m(H_i)}S$.  Furthermore, for a flat $Y\in L$ of rank two let $J(Y)=\sum_{Y\subset H} J(H)$.  

\subsection{Hilbert-Burch resolutions}\label{ss:HB} For later use, we will need to have a good understanding of the syzygies of the ideals $J(Y_i)$, which we write explicitly below:
\[
\begin{array}{c}
J(Y_1)=\langle x^{m_1},y^{m_2},(x-\alpha y)^{m_4}\rangle\\
J(Y_2)=\langle x^{m_1},z^{m_3},(x+z)^{m_5}\rangle\\
J(Y_3)=\langle y^{m_2},z^{m_3},(y+z)^{m_6}\rangle.
\end{array}
\]
Each of these ideals is visibly codimension two (being supported at the relevant triple point in $\mathbb{P}^2$) and is Cohen-Macaulay.  As such, each ideal $J(Y_i)$ has a \textit{Hilbert-Burch resolution} (see~\cite[Theorem~3.2]{E05}) of the form
\[
0\rightarrow S^2\xrightarrow{\phi_i} S^3 \rightarrow J(Y_i)\rightarrow 0,
\]
where $\phi_i$ is a $2\times 3$ matrix of forms whose $2\times 2$ minors are (up to multiplication by a non-zero constant) the given generators of $J(Y_i)$.  The columns of $\phi_i$ are \textit{syzygies} on the generators of $J(Y_i)$.  We will write
\[
\phi_1=
\begin{bmatrix}
A_1 & B_1\\
A_2 & B_2\\
A_3 & B_3
\end{bmatrix}
\qquad
\phi_2=
\begin{bmatrix}
C_1 & D_1\\
C_2 & D_2\\
C_3 & D_3
\end{bmatrix}
\qquad
\phi_3=
\begin{bmatrix}
E_1 & F_1\\
E_2 & F_2\\
E_3 & F_3
\end{bmatrix}.
\]
To make the meaning of these matrices clearer, consider the ideal $J(Y_1)=\langle x^{m_1},y^{m_2},(x-\alpha y)^{m_4}\rangle$.  Then the columns of $\phi_1$ express the relations $A_ix^{m_1}+B_iy^{m_2}+C_i(x-\alpha y)^{m_4}=0$ for $i=1,2$ (notice the relative order of the generators is important).  Since these are matrices for Hilbert-Burch resolutions, we also have (up to multiplication by a non-zero constant)
\[
A_1B_2-B_1A_2=(x-\alpha y)^{m_4}, A_1B_3-B_1A_3=y^{m_2}, A_2B_3-B_2A_3=x^{m_1},\mbox{etc.}
\] 
It is standard practice in commutative algebra to omit redundant generators from the generating set.  However, it is important for our analysis that we fix the given generators for the ideals $J(Y_i)$.  We illustrate with two simple examples.
\begin{exm}\label{ex:HB}
Suppose $m_1=m_2=2$ and $m_4=1$, so $J(Y_1)=\langle x^2,y^2,(x-\alpha y)\rangle$.  Clearly either $x^2$ or $y^2$ is an extraneous generator.  There are several possible choices for what $\phi_1$ will look like, depending on what syzygies are chosen on $J(Y_1)$.  We make the following choice:
\[
\phi_1=
\begin{bmatrix}
1 & 0\\
-\alpha^2 & -(x-\alpha y)\\
-(x+\alpha y) & y^2
\end{bmatrix},
\]
which expresses the fact that syzygies on $J(Y_1)$ can be generated by the Koszul syzygy between $y^2$ and $(x+y)$, followed by the expression of $x^2$ in terms of $y^2$ and $x-\alpha y$.  Notice that the $2\times 2$ minors are $-(x-\alpha y),y^2,$ and $-x^2$.  More generally, suppose $J(Y_1)$ is minimally generated by two of the three forms $x^{m_1},y^{m_2},(x-\alpha y)^{m_4}$; without loss suppose $y^{m_2}$ and $(x-\alpha y)^{m_4}$ generate $J(Y_1)$.  Then $\phi_1$ has the form
\[
\phi_1=
\begin{bmatrix}
1 & 0\\
-F & -(x-\alpha y)^{m_4}\\
-G & y^{m_2}
\end{bmatrix},
\]
where $F$ and $G$ are polynomials so that $Fy^{m_2}+G(x-\alpha y)^{m_4}=x^{m_1}$.

Now suppose instead that $m_1=m_2=m_4=2$, so $J(Y_1)=\langle x^2,y^2,(x-\alpha y)^2\rangle$.  This time none of the generators are redundant.  Then one choice for $\phi_1$ is:
\[
\phi_1=
\begin{bmatrix}
x-2\alpha y & y\\
\alpha^2x & -2\alpha x+\alpha^2y\\
-x & -y
\end{bmatrix}.
\]
Notice that the $2\times 2$ minors are $-2\alpha(x+y)^2,2\alpha y^2,$ and $-2\alpha x^2$.
\end{exm}

\subsection{The chain complex}
We will now assemble the promised chain complex whose homologies will tell us about freeness of the multi-arrangement $(X_3,\m)$.  First consider the chain complex $\cS=\cS_0\xrightarrow{\delta^0}\cS_1\xrightarrow{\delta^1}\cS_2$ with modules
\[
\cS_0\cong S^3 \qquad \cS_1\cong S^6 \qquad \cS_2\cong S^3.
\]
We will consider taking homologies beginning at the left-hand side, so will denote homologies as \textit{co}homologies.  The free module $\cS_0$ has a basis in correspondence with the variables $x,y,z$, $\cS_1$ has basis in correspondence with hyperplanes (codimension one flats), and $\cS_2$ has basis in correspondence with the triple points of $X_3$.  The maps $\delta^i:\cS_{i}\rightarrow \cS_{i+1}$ are given by the matrices
\[
\delta^0=\bordermatrix{ & x & y & z\cr 
H_1 & 1 & 0 & 0 \cr
H_2 & 0 & 1 & 0 \cr
H_3 & 0 & 0 & 1 \cr
H_4 & 1 & -\alpha & 0 \cr
H_5 & 1 & 0 & 1 \cr
H_6 & 0 & 1 & 1}
\qquad
\delta^1=\bordermatrix{ & H_1 & H_2 & H_3 & H_4 & H_5 & H_6 \cr
Y_1 & 1 & -\alpha & 0 & -1 & 0 & 0 \cr
Y_2 & 1 & 0 & 1 & 0 & -1 & 0 \cr
Y_3 & 0 & 1 & 1 & 0 & 0 & -1
}.
\] 
Note that $\delta^0$ is the coefficient matrix of forms defining $X_3$ and $\delta^1$ is the matrix of relations around triple points. It is readily checked that the chain complex $\cS$ is exact.  In fact, it can be checked that the homologies of $\cS$ govern when $X_3$ is $2$-formal in the sense of Brandt and Terao~\cite{BT94} (see also~\cite{T07}); as a consequence we see that $X_3$ is $2$-formal.

Now define the sub-complex $\J=\J_1\xrightarrow{\delta^1}\J_2$ of $\cS$ with modules
\[
\J_1=\bigoplus\limits_{H\in L_1} J(H)\subset \cS_1\qquad \J_2=\bigoplus\limits_{Y\in \lt}J(Y)\subset \cS_2.
\]
The quotient complex $\cS/\J$ has the form
\[
S^3\xrightarrow{\bar{\delta}^0} \bigoplus\limits_{H\in L_1} \dfrac{S}{J(H)} \xrightarrow{\bar{\delta}^1}  \bigoplus\limits_{Y\in \lt}\dfrac{S}{J(Y)},
\]
where $\bar{\delta}^0,\bar{\delta}^1$ are the quotient maps.  Notice that $D(X_3,\m)$ is the kernel of the map $\bar{\delta}^0$, in other words $D(X_3,\m)\cong H^0(\cS/\J)$.

\begin{remark}
The modules of the chain complex $\cS/\J$ are a natural extension to multi-arrangements of the modules $\bigoplus_{Y\in L_c} D_c(\A_Y)$ which appear in the paper of Brandt and Terao~\cite{BT94}.
\end{remark}

The main point of the chain complex $\cS/\J$ is that its homologies control freeness of $D(X_3,\m)$.  First, we notice that we may consider instead homologies of the complex $\J$.

\begin{prop}\label{prop:HomologySwitch}
With the chain complexes $\J,\cS,\cS/\J$ as above, $D(X_3,\m)\cong H^0(\cS/\J)\cong H^1(\J)$, $H^1(\cS/\J)\cong H^2(\J)$, and $H^2(\cS/\J)=0$.
\end{prop}
\begin{proof}
This follows directly from the long exact sequence in homology derived from the short exact sequence of complexes $0\rightarrow \J\rightarrow \cS\rightarrow \cS/\J \rightarrow 0$, the fact that the homologies of $\cS$ vanish, and the fact that $D(X_3,\m)\cong H^0(\cS/\J)$.
\end{proof}

The main result of this section is the following proposition.
\begin{prop}\label{prop:VanishingHomologies}
$D(X_3,\m)$ is a free $S$-module if and only if $H^2(\J)=0$.	
\end{prop}
\begin{proof}
Let $\frakm=\langle x,y,z\rangle$ be the maximal ideal of $S$.  First we claim that if $H^2(\J)$ is nonzero, then it is only supported at $\frakm$ (hence has finite length).  It is easiest to show this using $H^1(\cS/\J)$; by Proposition~\ref{prop:HomologySwitch}, $H^1(\cS/\J)\cong H^2(\J)$.  We will show that the localization $H^1(\cS/\J)_P$ at any homogeneous non-maximal prime is zero, using the fact that localization commutes with taking homology.  To this end, localize the complex $\cS/\J$:
\[
0\rightarrow S^3_P\xrightarrow{(\bar{\delta}^0)_P} \bigoplus\limits_{H\in L_1} \dfrac{S}{J(H)}_P \xrightarrow{(\bar{\delta}^1)_P}\bigoplus\limits_{Y\in \lt}\dfrac{S}{J(Y)}_P \rightarrow 0.
\]
If $J(H)\not\subset P$ or $J(Y)\not\subset P$, then the corresponding summand in the above chain complex vanishes.  Suppose $P$ is codimension one.  Then $P$ is principal, hence $P$ can contain $J(H)$ for at most one $H\in L_1$.  If $P$ contains none of the ideals $J(H)$, for $H\in L_1$, then clearly $H^1(\cS/\J)_P=0$.  So suppose there is some $H\in L_1$ so that $J(H)\subset P$ (since $P$ is codimension one, this must be the only such $J(H)$).  Clearly in this case $(\bar{\delta}_0)_P:S^3_P\rightarrow S/J(H)_P$ is surjective, so again $H^1(\cS/\J)_P=0$.  Now suppose $P$ is codimension two (so $P$ is the ideal of a point in projective two-space).  If $P$ contains only one $H\in L_1$, then the same argument as above shows that $H^1(\cS/\J)_P=0$.  If $P$ contains two or more hyperplanes of $\A$, then $P$ must be the ideal of the point of intersection of whatever hyperplanes $H$ satisfy $J(H)\subset P$.  If the point is a simple point of $\A$, then there are only two hyperplanes $H,H'\in\A$ so that $J(H)\subset P$ and $J(H')\subset P$.  So $(\cS/\J)_P$ looks like
\[
0\rightarrow S^3_P\xrightarrow{(\bar{\delta}^0)_P} \dfrac{S}{J(H)}_P\oplus\dfrac{S}{J(H')}_P,
\]
which is clearly exact on the right, so $H^1(\cS/\J)_P=0$.  Finally, suppose $P$ is the ideal of a triple point $Y\in \lt$ (intersection of $H,H',$ and $H''\in\A$).  Then $(\cS/\J)_P$ looks like
\[
0\rightarrow S^3_P\xrightarrow{(\bar{\delta}^0)_P} \dfrac{S}{J(H)}_P\oplus\dfrac{S}{J(H')}_P\oplus\dfrac{S}{J(H'')}_P\xrightarrow{(\bar{\delta}^1)_P} \dfrac{S}{J(Y)}_P \rightarrow 0.
\]
A check yields that $\mbox{coker}((\bar{\delta}_0)_P)=(S/J(Y))_P$, so again $H^1(\cS/\J)_P=0$.  It follows that $H^1(\cS/\J)$ is only supported at $\frakm$, the homogeneous maximal ideal.

Now to show that freeness of $D(X_3,\m)$ is equivalent to vanishing of $H^2(\J)$, consider the following four-term exact sequence:
\begin{equation}\label{D-exact}
0\rightarrow D(X_3,\m)\cong H^1(\J)\rightarrow  \bigoplus\limits_{H\in L_1} J(H)\xrightarrow{\delta^1} \bigoplus\limits_{Y\in \lt}J(Y) \rightarrow H^2(\J) \rightarrow 0.
\end{equation}
Write $\J_1$ for $\bigoplus\limits_{H\in L_1} J(H)$ and $\J_2$ for $J(Y_1)\oplus J(Y_2) \oplus J(Y_3)$.  Then break this four term sequence into the two short exact sequences
\[
0\rightarrow K \rightarrow \J^2 \rightarrow H^2(\J) \rightarrow 0\qquad\mbox{and}\qquad 0\rightarrow D(X_3,\m) \rightarrow \J^1 \rightarrow K \rightarrow 0,
\]
where $K=\mbox{im}(\delta^1)$.  By the long exact sequence in $\mbox{Ext}$ applied to the first short exact sequence, coupled with the fact that $\mbox{Ext}^j_S(J(Y_i),S)$ vanishes when $j>1$ (hence $\mbox{Ext}^j_S(\J_2,S)$ vanishes when $j>1$), we obtain $\mbox{Ext}^3_S(H^2(\J),S)\cong\mbox{Ext}^2_S(K,S)$.  Similarly, applying the long exact sequence in $\mbox{Ext}$ to the second short exact sequence (and using that $\mbox{Ext}^j(\J_1,S)=0$ for $j>0$ since $\J_1$ is a sum of principal ideals) yields that $\mbox{Ext}^2_S(K,S)\cong \mbox{Ext}^1(D(X_3,\m),S)$.  All in all, we see that $\mbox{Ext}^3_S(H^2(\J),S)\cong \mbox{Ext}^1_S(D(X_3,\m),S)$.  Since $D(X_3,\m)$ is a second syzygy, freeness of $D(X_3,\m)$ is equivalent to vanishing of $\mbox{Ext}^1_S(D(X_3,\m),S)$, which is equivalent to vanishing of $\mbox{Ext}^3_S(H^2(\J),S)$.  Since $H^2(\J)\cong H^1(\cS/\J)$ is only supported at $\frakm$, $\mbox{Ext}^3_S(H^2(\J),S)$ vanishes if and only if $H^2(\J)=0$, and we are done.
\end{proof}

\section{Free multiplicities}\label{sec:freemult}

In this section we derive the full characterization of free multiplicities on $(X_3,\m)$.  The characterization hinges on the use of Proposition~\ref{prop:VanishingHomologies}, together with a precise description of the homology module $H^2(\J)$.  We begin by giving this description in terms of syzygies.

Denote by $e_H$ a generator for $J(H)$ (of degree $\m(H)$) and given a triple point $Y$ and a hyperplane $H$ passing through $Y$ we will denote by $e_{H,Y}$ (of degree $\m(H)$) a generator for $J(H)$ viewed as a sub-ideal of $J(Y)$.  We will use the following commutative diagram to get a presentation for $H^2(\J)$ using syzygies of the ideal $J(Y)$.

\begin{center}
\begin{tikzcd}
	0 & 0 \\
	\bigoplus\limits_{H\in L_1} J(H) \ar{r}{\delta^1}\ar{u} &\bigoplus\limits_{Y\in \lt} J(Y) \ar{u} \\
	\bigoplus\limits_{H\in L_1} S[e_{H}] \ar{r}{\hat{\delta}^1}\ar{u} & \bigoplus\limits_{Y\in \lt}\left[\bigoplus\limits_{\substack{H\in L_1\\ Y\subset H}} S[e_{H,Y}] \right] \ar{u}\\
	0\ar{u}\ar{r} &\bigoplus\limits_{Y\in \lt} \syz(J(Y)) \ar{u}\\
	& 0\ar{u} \\
\end{tikzcd}
\end{center}

The snake lemma yields an exact sequence
\begin{equation}\label{eq:SNAKE}
\bigoplus\limits_{Y\in \lt} \syz(J(Y))\rightarrow \mbox{coker}(\hat{\delta}^1)\rightarrow \mbox{coker}(\delta^1)=H^2(\J)\rightarrow 0.
\end{equation}
We now explicitly identify the image of the syzygy modules inside of $\mbox{coker}(\hat{\delta}^1)$, which we will show is isomorphic to $S^3$.  As noted in \S~\ref{ss:HB}, each of the ideals $J(Y_i)$ is Cohen-Macaulay of codimension two, with Hilbert-Burch resolution of the form
\[
0\rightarrow S^2\xrightarrow{\phi_i} S^3 \rightarrow J(Y_i)\rightarrow 0,
\]
where, as in \S~\ref{ss:HB}, we will write
\[
\phi_1=
\begin{bmatrix}
A_1 & B_1\\
A_2 & B_2\\
A_3 & B_3
\end{bmatrix}
\qquad
\phi_2=
\begin{bmatrix}
C_1 & D_1\\
C_2 & D_2\\
C_3 & D_3
\end{bmatrix}
\qquad
\phi_3=
\begin{bmatrix}
E_1 & F_1\\
E_2 & F_2\\
E_3 & F_3
\end{bmatrix}.
\]
We see that the image of $\bigoplus_{Y\in L_2} \syz(J(Y))$ inside of $\bigoplus\limits_{Y\in \lt}\left[\bigoplus\limits_{\substack{H\in L_1\\ Y\subset H}} S[e_{H,Y}] \right]$ is generated by the columns of the following matrix (we label the rows by pairs $(H_i,Y_i)$ corresponding to $e_{H_i,Y_i}$:
\[
\bordermatrix{ & & & & & \cr
[H_1,Y_1] & A_1 & B_1 & 0 & 0 & 0 & 0 \cr 
[H_2,Y_1] & A_2 & B_2 & 0 & 0 & 0 & 0 \cr 
[H_4,Y_1] &	A_3 & B_3 & 0 & 0 & 0 & 0 \cr 
[H_1,Y_2] & 0   &   0 &C_1&D_1& 0 & 0 \cr 
[H_3,Y_2] & 0   & 0   &C_2&D_2& 0 & 0 \cr
[H_5,Y_2] & 0   & 0   &C_3&D_3& 0 & 0 \cr
[H_2,Y_2] & 0   & 0   & 0 & 0 &E_1&F_1\cr 
[H_3,Y_3] & 0   & 0   & 0 & 0 &E_2&F_2\cr 
[H_6,Y_6] & 0   & 0   & 0 & 0 &E_3&F_3\cr 
	}_\text{\large{.}}
\]
The map $\hat{\delta}^1$ is lifted from $\delta^1$ and has the following matrix:
\[
\bordermatrix{& H_1 & H_2 &H_3&H_4&H_5&H_6\cr
	[H_1,Y_1] & 1   & 0   & 0 & 0 & 0 & 0 \cr 
	[H_2,Y_1] & 0   & -\alpha   & 0 & 0 & 0 & 0 \cr 
	[H_4,Y_1] &	0   & 0   & 0 &-1 & 0 & 0 \cr 
	[H_1,Y_2] & 1   &   0 &0  & 0 & 0 & 0 \cr 
	[H_3,Y_2] & 0   & 0   & 1 & 0 & 0 & 0 \cr
	[H_5,Y_2] & 0   & 0   & 0 & 0 &-1 & 0 \cr
	[H_2,Y_3] & 0   & 1   & 0 & 0 & 0 & 0 \cr 
	[H_3,Y_3] & 0   & 0   & 1 & 0 & 0 & 0 \cr 
	[H_6,Y_3] & 0   & 0   & 0 & 0 & 0 & -1\cr 
}_\text{\large{.}}
\]
This matrix clearly has full rank, so $\mbox{coker}(\hat{\delta}^1)\cong S^3$, generated for instance by the images of $e_{H_1,Y_1},e_{H_2,Y_1},$ and $e_{H_3,Y_2}$; call these $[e_{H_1,Y_1}],[e_{H_2,Y_1}],$ and $[e_{H_3,Y_2}]$.  With respect to this basis we may represent $\mbox{coker}(\hat{\delta}^1)$ as the image of $S^9$ under the projection with matrix
\[
\begin{bmatrix}
1 & 0 & 0 & -1 & 0 & 0 & 0 & 0 & 0 \\
0 & 1 & 0 & 0 & 0 & 0 & \alpha & 0 & 0 \\
0 & 0 & 0 & 0 & 1 & 0 & 0 & -1 & 0
\end{bmatrix}_\text{\large{.}}
\]
With this choice of projection, the image of the syzygies inside of $S^3$ is generated by the columns of the following matrix (rows are labeled by $[H_i,Y_1]$ corresponding to $[e_{H_i,Y_1}]$ for $i=1,2,4$):
\[
M=\bordermatrix{ & & & & & \cr 
	[H_1,Y_1] &A_1&B_1&-C_1&-D_1&    0      &      0      \cr
	[H_2,Y_1] &A_2&B_2& 0  & 0  &\alpha E_1&\alpha F_1  \cr
	[H_3,Y_2] & 0 & 0 &C_2 &D_2 &-E_2       &-F_2
	}_\text{\large{.}}
\]
Together with Proposition~\ref{prop:VanishingHomologies} and the presentation~\eqref{eq:SNAKE}, we have proved the following result.
\begin{prop}\label{prop:HomologyPresentation}
With all notation as above, $D(X_3,\m)$ is free if and only if the columns of $M$ generate the free module $S^3=\mbox{coker}(\hat{\delta}^1)$.
\end{prop}

Now we classify free multiplicities on $X_3$ (and all of its moduli).

\begin{thm}\label{thm:X3Classification}
For any $\alpha\neq 0,1$, freeness of $(X_3,\m)$ implies $m_4=m_5=m_6=1$ and $m_1=m_2=m_3=n$ for some positive integer $n$.  So free multiplicities are of the form $[n,n,n,1,1,1]$.  Furthermore,
\begin{itemize}
	\item if $\alpha$ is not a root of unity then $[n,n,n,1,1,1]$ is a free multiplicity for any $n>1$, and
	\item if $\alpha$ is a root of unity with order $m$, then $[n,n,n,1,1,1]$ is a free multiplicity if and only if $n\not\equiv 1\mbox{ mod } m$.  (In particular, if $\alpha=-1$, then $[n,n,n,1,1,1]$ is a free multiplicity if and only if $n$ is even.)
\end{itemize}

Moreover, when $(X_3,\m)$ is free the exponents are $(n+1,n+1,n+1)$.
\end{thm}

\begin{remark}
In Theorem~\ref{thm:X3Classification}, if we fix a multiplicity $\m=[n,n,n,1,1,1]$ then $(X_3,\m)$ is free as long as $\alpha$ avoids the roots of $t^{n-1}-1$.  Moreover $(X_3,\m)$ is not free for any $\alpha$ if $\m$ does not have the form $[n,n,n,1,1,1]$.  So among the moduli of arrangements with the $X_3$ lattice and a fixed multiplicity $\m$, the multi-arrangements which are free form a (possibly empty) Zariski open set.  We do not know if this is true in general.  Yuzvinsky~\cite{Y93} has shown that this is the case for simple arrangements.
\end{remark}

%\begin{remark}
%We thank Alex Fink for the suggestion to simplify the statement of Theorem~\ref{thm:X3Classification} (and the proof) by considering the order of $-\alpha$ instead of $\alpha$.
%\end{remark}

\begin{proof}
We use Proposition~\ref{prop:HomologyPresentation}.  The columns of $M$ generate $S^3$ if and only if there is a $3\times 3$ minor $M'$ of $M$ so that $\det(M')$ is a non-zero constant.  Let us see what constraints this assumption places on the entries of $M'$.

Suppose first that $M'$ contains two columns corresponding to the same ideal; without loss of generality assume that
\[
M'=
\left(
\begin{array}{ccc}
A_1 & B_1 & F \\
A_2 & B_2 & G \\
0   &  0  & H
\end{array}
\right),
\]
where the third column is some other column of $M$.  Then $\det(M')=H(A_1B_2-B_1A_2)$.  However, by the structure of the Hilbert-Burch matrix (see~\S~\ref{ss:HB}) $A_1B_2-B_1A_2$ is a (non-zero) constant multiple of a generator of the ideal $J(Y_1)$, hence this implies $\det(M')$ is either zero or has positive degree, so cannot be non-vanishing constant.  It follows that if $\det(M')$ is a non-zero constant, it must have a column corresponding to a single syzygy from each of the three ideals $J(Y_1),J(Y_2),$ and $J(Y_3)$.  Without loss of generality, we may assume that $M'$ has the form
\[
M'=
\left(
\begin{array}{ccc}
A_1 &-C_1   & 0\\
A_2 & 0   & \alpha E_1 \\
 0  &C_2  & -E_2
\end{array}
\right),
\]
with $\det(M')=-\alpha A_1E_1C_2-A_2C_1E_2$.  Suppose that one of the entries off of the anti-diagonal of $M'$ vanishes; without loss suppose $A_1=0$.  Then it follows that $A_2y^{m_2}+A_3(x+y)^{m_4}$ is a Koszul syzygy on the forms $y^{m_2}$ and $(x+y)^{m_4}$ (see Example~\ref{ex:HB}).  However, this implies that the second term $A_2C_1E_2$ either has positive degree or vanishes.  Both cases contradict that $\det(M')$ is a non-zero constant.  It follows that all entries other than the anti-diagonal entries of $M'$ are non-zero.  Therefore, since both terms of $\det(M')$ are homogeneous, the only way that $\det(M')$ is a non-zero constant is if all entries of $M'$ other than the anti-diagonal entries are non-zero constants.

Now recall the analysis in Example~\ref{ex:HB}.  The presence of \textit{any} non-zero constant in the matrix $\phi_i$ expresses the fact that $J(Y_i)$ is minimally generated by two of the three given generators.  For simplicity let us consider $J(Y_1)$.  The fact that $A_1$ and $A_2$ are \textit{both} non-zero constants means that $J(Y_1)$ may be generated by $(x-\alpha y)^{m_4}$ along with \textit{either} $x^{m_1}$ or $y^{m_2}$.  It is not difficult to check that $J(Y_1)$ can be generated by $y^{m_2}$ and $(x-\alpha y)^{m_4}$ if and only if $m_2+m_4\le m_1+1$.  Likewise, $J(Y_1)$ can be generated by $x^{m_1}$ and $(x-\alpha y)^{m_4}$ if and only if $m_1+m_4\le m_2+1$.  To satisfy both of these inequalities, we must have $m_4=1$ and $m_1=m_2$.  A similar analysis with $J(Y_2)$, $J(Y_3)$ yields $m_5=1$ and $m_1=m_3$, $m_6=1$ and $m_2=m_3$, respectively.  Hence the ideals of the triple points have the form
\[
J(Y_1)=\langle x^n, y^n, x-\alpha y\rangle \qquad J(Y_2)=\langle x^n,z^n,x+z\rangle \qquad J(Y_3)=\langle y^n,z^n,y+z\rangle,
\]
where $m_1=m_2=m_3=n$.  This proves the first statement of the theorem.  Now, if we assume $m_1=m_2=m_3=n$, we have
\[
\phi_1=
\left(
\begin{array}{cc}
 1 & 0 \\
-\alpha^{n} & -(x-\alpha y) \\
-(x^{n}-\alpha^{n}y^{n})/(x-\alpha y) & y^{n}
\end{array}
\right),
\]
\[
\phi_2=
\left(
\begin{array}{cc}
1 & 0 \\
(-1)^{n+1} & -(x+z) \\
-(x^{n}-(-1)^{n}z^{n})/(x+ z) & z^{n}
\end{array}
\right),
\]
\[
\mbox{and}\qquad\phi_3=
\left(
\begin{array}{cc}
1 & 0 \\
(-1)^{n+1} & -(y+z) \\
-(y^{n}-(-1)^{n}z^{n})/(y+z) & z^{n}
\end{array}
\right).
\]
See Example~\ref{ex:HB} for details on the computation of $\phi_1,\phi_2,$ and $\phi_3$.  This yields $A_1=C_1=E_1=1,A_2=-\alpha^n$, and $C_2=E_2=(-1)^{n+1}$, so $\det M'=-\alpha A_1E_1C_2-A_2C_1E_2=-\alpha (-1)^{n+1}(\alpha^{n-1}-1)$.  This is zero (hence $[n,n,n,1,1,1]$ is not free) if and only if $\alpha^{n-1}=1$; equivalently $n\equiv 1 \mod m$, where $m$ is the order of $\alpha$.

Finally, the exponents for the free multiplicities can be calculated by two different methods. First, one can look at the short exact sequence in \eqref{D-exact} and compute the Hilbert functions to get the exponents. Second, one can use the multiarrangement addition-deletion theorem \cite{ATW08} with the deleted hyperplane being $x-\alpha y$. In this case the deletion is the deleted $A_3$ arrangement whose free multiplicities are classified by Takuro Abe in \cite{A07}. 
\end{proof}

\begin{remark}\label{rem:Characteristic}
Theorem~\ref{thm:X3Classification} appears to hold over finite fields (of characteristic greater than two, since we assume $\alpha\neq 0,1$).
\end{remark}

%ADded by Mike 06/19

%\begin{remark}

%The case $\alpha=1$ presents a negative answer to a question presented by Takuro Abe at the end of the introduction of \cite{A07}.

%\end{remark}

\begin{remark}
	It is possible to recover portions of Theorem~\ref{thm:X3Classification} using other techniques for multi-arrangements.  We list a few here.
	\begin{enumerate}
		\item If $\alpha=-1$ and we assume $m_4=m_5=m_6=1$, then local and global mixed products~\cite{ATW07} can be used to show that $(X_3,\m)$ is free if and only if $m_1=m_2=m_3=2k$ for some natural number $k$.
		\item If $(X_3,\m)$ satisfies that all sub-$A_2$ multi-arrangements are `balanced,' then local and global mixed products show that $\m$ is not a free multiplicity if one of $m_4,m_5,$ or $m_6$ is greater than one.
		\item If $m_i=1$ for some $i=1,\ldots,6$, then the deletion of $(X_3,\m)$ with respect to $H_i$ has the deleted $A_3$ arrangement as the base arrangement.  The classification of all free multiplicities on the deleted $A_3$ arrangement by Abe~\cite{A07}, along with addition/deletion/restriction theorems for multi-arrangements~\cite{ATW08} yields some constraints on the $m_i$ (however, this method runs into the problem of Remark~\ref{rem:P1}).
		\item If $(X_3,\m)$ has a `heavy flag,' then $\m$ is not a free multiplicity~\cite[Corollary~5.2]{AK16}.
	\end{enumerate}
	We do not know if it is possible to recover Theorem~\ref{thm:X3Classification} entirely using other (non-homological) techniques for multi-arrangements.
\end{remark}

%\begin{remark}\label{multi-add}

%The free multiplicities of all moduli on $X_3$ are in some sense \emph{isolated} in the sense of \cite{AN12}. The only deletion we can take from %a free multiplicity on $X_3$ and get a free multiarrangement is where we restrict to one of the multiplicity 1 lines $x+\alpha y$, $x+z$, or $y+z$. %We will investigate this further in the next section.

%\end{remark}

Next we examine a basis for the case where $\alpha=-1$ and $D(X_3,\m)$ is free.

\begin{prop}\label{lem:Basis}
The derivations \[
\begin{array}{rl}
\theta_1 & =x^{2k+1}\dd_1+y^{2k+1}\dd_2+z^{2k+1}\dd_3\\
\theta_2 & =(y+z)(x^{2k}\dd_1-y^{2k}\dd_2-z^{2k}\dd_3)\\
\theta_3 & =x^{2k}z\dd_1-(x+y+z)y^{2k}\dd_2+xz^{2k}\dd_3
\end{array}
\] form a basis for the $S$-module $D(X_3,\m)$ for $\alpha=-1$ where $\m=[2k,2k,2k,1,1,1]$. 
\end{prop}
\begin{proof} This is a routine check with Saito's criterion \cite{S80}. \end{proof}

\section{Free extensions}\label{exams}

Given an arrangement $\A'$ of rank $r$, an \textit{extension} of $\A'$ is an arrangement $\A$ of rank $r+1$ so that there exists a hyperplane $H_0\in\A$ with $\A^{H_0}=\A'$.  In our case, where $\A'$ is defined by linear forms in $S=\kk[x,y,z]$, we will assume $\A$ is defined by forms in $\kk[x,y,z,w]$ and $H_0$ is defined by $w=0$.  A \textit{free extension} of a multi-arrangement $(\A',\m)$ is a free arrangement $\A$ with a hyperplane $H_0\in\A$ so that the Ziegler multi-restriction satisfies $\A^{H_0}=(\A',\m)$.  The free multiplicities of the $X_3$ arrangement do admit free extensions; we write these down in Proposition \ref{prop:Extensions}, but first we need a preliminary lemma.

\begin{lem}\label{lem:LocalAlongW}
Let $n>0$ be an integer and let $\A$ be the central arrangement given by the vanishing of forms
\[
\begin{array}{c}
x-a_1z,\ldots,x-a_nz,\\
y-b_1z,\ldots,y-b_nz,\\
Ax+By+Cz,z,
\end{array}
\]
where $a_1,\ldots,a_n,b_1,\ldots, b_n,A,B,C$ are all constants with $A,B\neq 0$.  Then $\A$ is free if and only if there is a re-ordering of the $a_i$ so that $(a_1,b_1,1),\ldots,(a_n,b_n,1)$ all lie on the line $Ax+By+Cz=0$ (in $\mathbb{P}^2(\kk)$).
\end{lem}
\begin{proof}
We use Yoshinaga's criterion for freeness of $3$-arrangements~\cite{Y05}.  Namely, a three-arrangement is free if and only if 
\begin{enumerate}
\item $\chi(\A,t)=(t-1)(t-d_1)(t-d_2)$ and
\item The Ziegler restriction $(\A^H,\m)$ is free with exponents $(d_2,d_2)$.
\end{enumerate}
Restricting to the hyperplane $H=V(z)$, we get the multi-arrangement $\A^H=V(x)\cup V(y)\cup V(Ax+By)$ with multiplicities $\m=[n,n,1]$.  By a characterization due to Wakamiko~\cite{W07}, $(\A^H,\m)$ is free with exponents $(n,n+1)$.  Suppose that the line $Ax+By+Cz=0$ passes through $q$ of the $n^2$ intersection points of the grid defined by $x_1=a_1z,\ldots,x_n=a_nz,y_1=b_1z,\ldots,y_n=b_nz$.  Then we compute
\[
\chi(\A,t)=(t-1)(t^2-(2n+1)t+n^2+2n-q),
\]
which factors as $(t-1)(t-n)(t-(n+1))$ if and only if $q=n$.  Since $A,B\neq 0$, the re-ordering in the statement of the lemma is possible.
%This is possible if and only if the line $Ax+By+Cz=0$ passes through $k$ of the points $(a_i,b_j,1)$ (otherwise $q>k$ since $Ax+By+Cz=0$ meets every line in a single point).
\end{proof}

\begin{prop}\label{prop:Extensions}
Let $n>1$ be an positive integer, $\m_n=[n,n,n,1,1,1]$, and $\alpha\neq 0,1\in\kk$.  If $(X_3,\m_n)$ has a free extension then $\alpha$ is a root of unity and $n=mt$ where $m$ is the order of $\alpha$ and $t$ is a positive integer.  Suppose $A_1,\ldots,A_t\in\kk^*$ are nonzero constants.  Then the arrangement defined by the vanishing of the forms
\[
\begin{array}{ccc}
x-A_1w,x-\alpha A_1w,\ldots,x-\alpha^{m-1}A_1w,\\
\vdots\\
x-A_tw,x-\alpha A_tw,\ldots,x-\alpha^{m-1}A_tw,\\
y-A_1w,y-\alpha A_1w,\ldots,y-\alpha^{m-1}A_1w,\\
\vdots\\
y-A_tw,y-\alpha A_tw,\ldots,y-\alpha^{m-1}A_tw,\\
z+A_1w,z+\alpha A_1w,\ldots,z+\alpha^{m-1}A_1w\\
\vdots\\
z+A_tw,z+\alpha A_tw,\ldots,z+\alpha^{m-1}A_tw\\
x-\alpha y,x+z,y+z,w
\end{array}
\]
is a free extension of $(X_3,\m_n)$.  Moreover, up to translations, every free extension has this form.
\end{prop}
\begin{proof}
Freeness of the arrangement $\A$ given by the indicated forms can be verified by applying Yoshinaga's criterion~\cite{Y04}; since $(X_3,\m_n)$ is free by Theorem~\ref{thm:X3Classification}, we check that $\A$ is locally free along $w=0$.  This reduces to checking freeness of the rank three closed sub-arrangements of $\A$ along $w=0$; all of these are free by Lemma~\ref{lem:LocalAlongW}.

Now we show that any free extension must have this form.  By Theorem~\ref{thm:X3Classification} and Theorem~\ref{thm:Yosh}, a free extension $\A$ of $(X_3,\m)$ exists only if $\m=[n,n,n,1,1,1]$ for some integer $n>1$.  In this case $\A$ has the form
\[
\begin{array}{c}
x-a_1w,\ldots,x-a_nw\\
y-b_1w,\ldots,y-b_nw\\
z-c_1w,\ldots,z-c_nw\\
x-\alpha y+Aw,x+z+ Bw,y+z+Cw,w,
\end{array}
\]
for some constants $a_1,\ldots,a_n,b_1,\ldots,b_n,c_1,\ldots,c_n,A,B,$ and $C$.  Translating in the $x,y,$ and $z$ directions we may assume that $A=B=C=0$.  We now check local freeness along $w=0$.  Let $\A_{xyw},\A_{xzw},\A_{yzw}$ be the closed rank three sub-arrangements consisting of hyperplanes which contain the line $x=y=w=0$, $x=z=w=0$, and $y=z=w=0$, respectively.  Projectively, these are all arrangements of the form considered in Lemma~\ref{lem:LocalAlongW}.  Assuming that $\A_{xzw}$ and $\A_{yzw}$ are both free, Lemma~\ref{lem:LocalAlongW} allows a re-indexing of the $a_i,b_i$ so that the points $(a_i,0,c_i,1)$ for $i=1,\ldots,n$ lie on the line determined by $x+z$ and the points $(0,b_i,c_i,1)$ for $i=1,\ldots,n$ lie along the line determined by $y+z$.  It follows that $a_i=b_i=-c_i$ for $i=1,\ldots,n$.

Now, by Lemma~\ref{lem:LocalAlongW}, $\A_{xyw}$ is free if and only if there are $n$ points among the $n^2$ points $(a_i,b_j,0,1)=(a_i,a_j,0,1)$ which lie along the line $x-\alpha y$.  Set $X=\{a_1,\ldots,a_n\}$ and consider the action of $\alpha$ on $X$ by multiplication.  We see that this action must be a permutation action.  Since $|X|=n>1$, $\alpha$ must be a root of unity; we denote the order of $\alpha$ by $m$.  Suppose $0\in X$.  Then the multiplication action of $\alpha$ on $X-\{0\}$ is faithful, so $|X-\{0\}|=tm$ (for some integer $t$) and $n=|X|=tm+1$.  But then $\m_n=[n,n,n,1,1,1]$ is not a free multiplicity by Theorem~\ref{thm:X3Classification} since $n\equiv 1 \mod m$.  It follows that $0\notin X$, so $\alpha$ acts faithfully on $X$ by multiplication, and $n=|X|=mt$ for some integer $t$.  Choosing representatives $A_1,\ldots,A_t$ in each $\alpha$-orbit gives the form in the statement of the proposition.
\end{proof}

\begin{remark}
Every free arrangement $\A$ in Proposition~\ref{prop:Extensions} has the property that $\A^{w=0}=X_3$ is not free.  For instance, taking $\alpha=-1$ and $t=1$ provides a one-parameter family of free arrangements of rank four with ten hyperplanes whose restriction is not free.  The first such example was a free arrangement of rank five with 21 hyperplanes provided by Edelman and Reiner~\cite{ER93}.
\end{remark}

\begin{remark}
Yoshinaga's extendability criterion~\cite{Y10} does not apply; the $X_3$ arrangement is locally $A_2$ but only has a positive system when $\alpha=-1$.  Furthermore, even if $\alpha=-1$, the $(*)$ criterion of~\cite[Theorem~2.5]{Y10} does not hold for the multiplicities $[n,n,n,1,1,1]$ when $n$ is even.
\end{remark}

%ADded by Mike 06/18

\begin{cor}
If $\A$ is an extension of $xyz(x-\alpha y)(x+z)(y+z)$ (for $\alpha\neq 0,1$), then it satisfies Terao's conjecture (see~\cite[Conjecture~4.138]{OT92}).  That is, freeness of $\A$ is determined from its intersection lattice.
\end{cor}

\begin{proof}
This is implicit in the proof of Proposition~\ref{prop:Extensions}.  Explicitly, all of the following steps are combinatorial:
\begin{itemize}
\item determining that the Ziegler multi-restriction has the multiplicity $[n,n,n,1,1,1]$ for $n>1$ (by Theorem~\ref{thm:X3Classification})
\item determining local freeness of $\A$ along $w=0$ (by Lemma~\ref{lem:LocalAlongW}), which implies $\alpha$ is a root of unity
\item determining that $n=mt$ for some $t$, where $m$ is the order of $\alpha$ (by Theorem~\ref{thm:X3Classification})
\item Finally, determining that $\A$ is free by Yoshinaga's criterion~\cite{Y04} and Theorem~\ref{thm:X3Classification}. \qedhere
\end{itemize}
\end{proof}

%\begin{remark}
%
%The deformation of X3 by titling a line (moduli space is dim 1) with multiplicities [2,2,2,1,1,1] does not have a free extension. Proof: write down possible extension
%
%\end{remark}

%--------------- Bibliography --------------%

\

\bibliographystyle{amsplain}

\bibliography{X3bib}

\end{document}